\documentclass[leqno,11pt]{article}

\usepackage{amsmath,amsthm,amssymb,mathrsfs} 
\usepackage{times}
\usepackage[pagebackref,hypertexnames=false]{hyperref}
\usepackage{fancybox,fancyhdr,graphics,epsfig}
\usepackage{algorithm}
\usepackage{algpseudocode}
\usepackage[usenames,dvipsnames]{color}
\usepackage{bbm}
\usepackage{verbatim}
\usepackage{extarrows}
\usepackage{caption}
\usepackage{subcaption}

\usepackage{authblk}

 \usepackage{amsmath,amsthm,bbm,amssymb}
\usepackage[dvipsnames]{xcolor}
\DeclareMathOperator{\birth}{birth}
\DeclareMathOperator{\death}{death}

\DeclareMathOperator{\vol}{Vol}
\DeclareMathOperator{\logg}{\log\log}

\def\R{\mathbb{R}}

\def\T{\mathbb{T}}

\def\cG{\mathcal{G}}

\def\cL{\mathcal{L}}

\def\cP{\mathcal{P}}

\def\cT{\mathcal{T}}

\def\cL{\mathcal{L}}

\newcommand{\E}{\mathbb{E}} 

\newcommand{\given}{\;|\;}
\newcommand{\mean}[1] {\E\left\{{#1}\right\}}

\newcommand{\meanx}[1] {\E\{{#1}\}}

\newcommand{\ind}{\boldsymbol{\mathbbm{1}}} 
\newcommand{\indf}[1]{\ind\set{#1}} 

\newcommand{\var}[1]{\mathrm{Var}\param{{#1}}}

\newcommand{\set}[1]{\left\{#1\right\}}

\newcommand{\param}[1]{\left(#1\right)}

\newcommand{\prob}[1]{\mathbb{P}\left(#1\right)}
\newcommand{\probx}[1]{\mathbb{P}(#1)}
\newcommand{\cprob}[2]{\mathbb{P}\left(#1\given #2\right)} 

\newcommand{\eps}{\epsilon}

\newcommand{\by}{\mathbf{y}}
\newcommand{\bx}{\mathbf{x}}

\providecommand{\setthms}[1]{#1}
\setthms{
\newtheorem{lem}{Lemma}[section]
\newtheorem{thm}[lem]{Theorem}
\newtheorem{prop}[lem]{Proposition}

\newtheorem{con}[lem]{Conjecture}
\newtheorem{rem}[lem]{Remark}

\theoremstyle{definition}

}

\newcommand{\ninf}{n\to\infty}

\newcommand{\limninf}{\lim_{\ninf}}
\definecolor{mygreen}{rgb}{0, 0.68, 0.31}
\definecolor{myred}{rgb}{1.0, 0,0}

\numberwithin{equation}{section}

\def\bsplit#1\esplit{\begin{split} #1 \end{split} }
\def\splitb#1\splite{\begin{split} #1 \end{split} }
\def\beq#1\eeq{\begin{equation} #1 \end{equation}}
\def\eqb#1\eqe{\begin{equation} #1 \end{equation}}

\def\dgm{\mathrm{dgm}}

\newcommand{\Dgm}{\mathrm{Dgm}}

\title{Cluster-Persistence for Weighted Graphs}
\author[1,2]{Omer Bobrowski}
\author[1,3]{Primoz Skraba}
\affil[1]{School of Mathematical Sciences, Queen Mary University of London, UK}
\affil[2]{Viterbi Faculty of Electrical and Computer Engineering, Technion, Israel}
\affil[3]{Department for Artificial Intelligence, Jozef Stefan Institute, Slovenia}

\date{}
\begin{document}

\maketitle

\begin{abstract}


 
Persistent homology is a natural tool for probing the topological characteristics of weighted graphs, essentially focusing on their $0$-dimensional homology. While this area has been substantially studied, we present a new approach to constructing a filtration for cluster analysis via persistent homology. The key advantages of the new filtration is that (a) it provides richer signatures for connected components by introducing non-trivial birth times, and (b) it is robust to outliers. The key idea is that nodes are ignored until they belong to sufficiently large clusters. We demonstrate the computational efficiency of our filtration, its practical effectiveness, and explore into its properties when applied to random graphs. 
\end{abstract}

\section{Introduction}

{\bf Clustering}  data is a fundamental task in unsupervised machine learning and exploratory data analysis. It has been the subject of countless studies over the last 50 years with many definitions and algorithms proposed, e.g., ~\cite{jain_data_1999,mcinnes_umap_2018}. 
{\bf Persistent homology} \cite{edelsbrunner_computational_2010,zomorodian_topology_2005} is a powerful topological tool that provides multi-scale structural information about data. Given an increasing sequence of spaces (filtration), persistent homology tracks the formation of connected components ($0$-dimensional cycles), holes ($1$-dimensional cycles), cavities ($2$-dimensional cycles), and their higher-dimensional extensions. The information encoded in persistent homology is often represented by a \emph{persistence diagram} -- a collection of points in $\R^2$, representing the birth and death of homology classes, and providing an intuitive numerical representation for topological information (see Figure \ref{fig:example_diagram}).
The connection between clustering and $0$-dimensional  persistent homology has been well-established  under a various different scenarios including the relationship with functoriality ~\cite{carlsson_classifying_2013,carlsson_characterization_2010}, and  density-based methods~\cite{bobrowski_topological_2017,chazal_persistence-based_2013}. An important motivating factor for connecting these methods is \emph{stability}. Namely,  given small perturbations of the input data, persistent homology and can provide  guarantees on the number of the output clusters. One important  drawback of this topological approach is that statistical tests for persistent homology and clustering based on persistence, have been lacking. 

Recently, for persistent homology in dimensions $1$ and above (i.e., excluding connected components), persistent homology based on distance filtrations was experimentally shown to exhibit a strong sense of universal behavior~\cite{bobrowski_universal_2023}. Suppose we are given as input a point-cloud generated by some unknown distribution. If we compute the distance-based persistent homology,
under an appropriate transformation, the  distribution of persistence values was shown to be independent of the original point-cloud distribution.
This phenomenon was then used to develop a statistical test to detect statistically significant homology classes.  A key point in \cite{bobrowski_universal_2023} is that in order to obtain such universal behavior, the measure of persistence is given by the value of $\death/\birth$, which makes the measure of persistence scale-invariant. 

However, in distance-based filtrations, the $0$-dimensional persistent homology (tracking clusters) does not fit into this universality framework, as the birth time of all the $0$-dimensional homology classes is set to $0$. 
To address this issue, and to enable the study of universality in the context of clustering, we introduce a new filtration, which we call the $k$-cluster filtration. This is a novel, non-local, construction, where vertices become ``alive'' only once they belong to a sufficiently large cluster. In other words, while traditional persistent homology considers every vertex as an individual clusters and tracks its evolution, in the $k$-cluster filtration we only consider components of $k$ or more vertices as `meaningful' clusters. 

We note that while the motivation for this new filtration is distance-based filtration, the $k$-cluster filtration can be constructed over any weighted graph. It generally provides two key advantages to the traditional filtration. Firstly, it results in a `richer' persistence diagram, in the sense that components have non-trivial birth times. This improves our ability to compare between the different features within the same diagram, or across different diagrams. Secondly, the $k$-cluster filtration provides a more `focused' view on connected components, by discarding those that are considered small (determined by application). In particular it easily allows to remove outliers from the persistence diagram.

The paper is organized as follows. Section \ref{sec:prelim} provides essential background about persistent homology. In Section \ref{sec:filt} we introduce the $k$-cluster filtration and give some preliminary properties. In Section \ref{sec:alg} we provide an algorithm for computing the filtration and corresponding persistence diagram in a single pass. In Section \ref{sec:exp} we show some experimental results comparing the clustering method to some other approaches. Finally, in Section \ref{sec:prob} we discuss some probabilistic aspects of this filtration, in comparison with known properties of random graphs and simplicial complexes.

\paragraph{Remark.} 
As we deal exclusively with $0$-dimensional homology, we phrase all the statements in this paper terms of weighted graphs rather than simplicial complexes. 

\section{Graph Filtrations and Persistent Homology}\label{sec:prelim}

In this section, we introduce the required topological notions. As we focus on the special case of graphs and connected components, i.e. 0-dimensional homology, we restrict our definitions to this case. For a general description of $k$-dimensional homology
we refer the reader to \cite{hatcher_algebraic_2002, munkres_elements_1984}.  

Let $G = (V,E)$ be an undirected  graph. Our main object of study is a \emph{graph filtration} or an increasing sequence of graphs.  
This can be constructed by defining a function $\tau:(V\cup E)\to [0,\infty)$, under the restriction that if $e=(u,v)\in E$, then $\tau(e) \ge \max(\tau(u),\tau(v))$. This restriction ensures that the sublevel sets of $\tau$ define a subgraph.
The filtration  $\set{\cG_t}_{t\ge 0}$ is then defined via
\[
\cG_t = \set{\sigma \in V\cup E : \tau(\sigma) \le t}.
\]
As we increase $t$ from $0$ to $\infty$, we can track connected components of $\cG_t$ as they appear and merge, which are referred to as \emph{birth} and \emph{deaths}, respectively. When two components merge, we use the `elder rule' to determine that the later born component is the one that dies. Note that at least one component has an infinite death time in any graph filtration. We refer the reader to \cite{edelsbrunner_computational_2010} for further details on this. 

These birth-death events can be tracked by an algebraic object called a $0$-dimensional persistent homology. Its most common visualization is via a \emph{persistence diagram} -- a collection of points in $\R^{2+}$, where each corresponds to a single connected component. The coordinates of a points encode the information with the $x$-coordinate representing the birth time, and the $y$-coordinate representing the death time. An example for a function on a line-graph is shown in Figure~\ref{fig:example_diagram}. Note that one component is  infinite which we  denote with a dashed line at the top of the diagram.  

In a more general context,
given a filtration of higher-dimensional objects (e.g., simplicial complexes), we can study the $k$-dimensional persistent homology. This object tracks the formation of $k$-dimensional cycles (various types of holes), and its definition is a natural extension of the $0$-dimensional persistent homology we study here. We refer the reader to \cite{edelsbrunner_computational_2010} for more information.

\begin{figure}[htbp]
 \centering
  \includegraphics[width=0.95\textwidth]{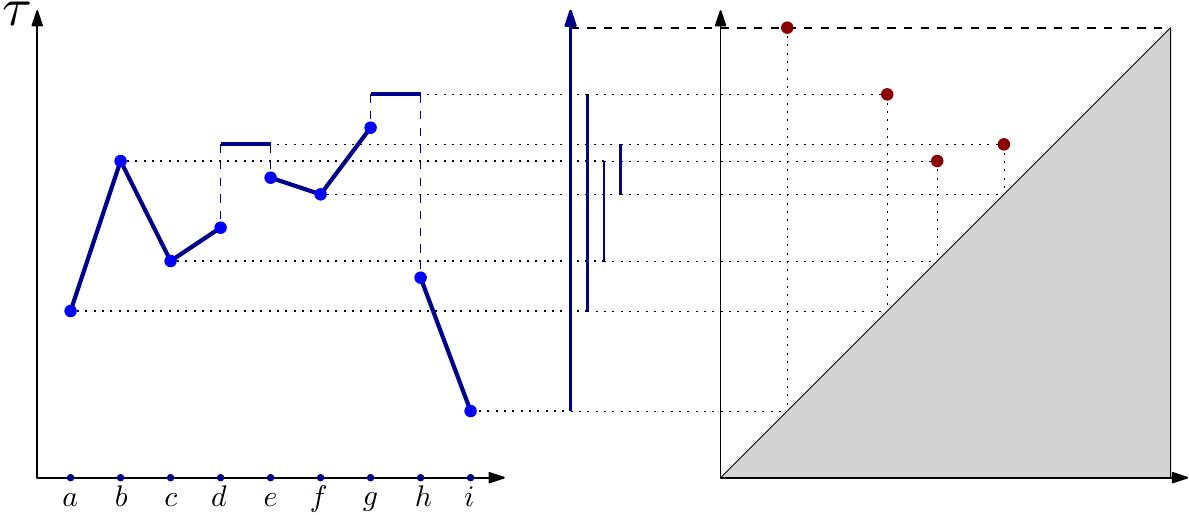}
    \caption{\label{fig:example_diagram}
    An example of a graph filtration on a line-graph. The filtration value of the vertices are given by $\tau$ (the $y$-axis). The filtration value of each edge is taken as the highest value between its plotted endpoints.    
    The bars in the middle represent the tracking of the components. The vertices which are local minima, i.e. $a$, $c$, $f$, and $i$, generate new components and so $\tau(a)$, $\tau(c)$, $\tau(f)$, and $\tau(i)$ correspond to birth times. The first merge occurs at $\tau(b) = \tau((a,b)) = \tau((b,c))$ merging 
    $\set{a}$ with $\set{c,d}$. In this case we declare the latter as dead since $\tau(a)<\tau(c)$. Next, at $\tau((d,e))$, the components $\set{a,b,c,d}$ and $\set{e,f}$ are merged, and the latter dies. Finally, at $\tau((g,h))$, the components $\set{a,b,c,d,e,f,g}$ and $\set{h,i}$ are merged, killing the former. The component containing $i$ has  the earliest birth time, and thus is declared  infinite.  }
\end{figure}

\section{The \texorpdfstring{$k$}{k}-Cluster Filtration}\label{sec:filt}

Let $G = (V,E,W)$ be an undirected weighted graph.
In computing $0$-dimensional persistent homology, the filtration values are commonly taken to be $\tau(v) = 0$ for all $v\in V$, and $\tau(e) = W(e)$ for all $e\in E$. We will denote this filtration by $\cG^*_t$. In other words, we assume all vertices are present at time zero, and edges are gradually added according to the weight function $W$. This has been the practice in the TDA literature in almost all studies, and in particular in the geometric settings where $W$ represents the distance between points (i.e., the geometric graph, which is the skeleton of both the \v Cech and Vietoris-Rips complexes). While in many models, this choice of $\tau$ seems reasonable, it has two significant drawbacks:
\begin{itemize}
    \item The produced persistence diagrams are  \emph{degenerate}, as the birth times of all 0-cycles is $t=0$. This significantly reduces the amount of information we can extract from persistence diagrams.
    \item The generated persistence diagrams are \emph{superfluous}, in the sense that they contains a point for each vertex $V$, while obviously not all vertices contribute significant structural information.
\end{itemize}
In this paper we propose a  modification to the standard graph filtration, that will resolve both of these issues, and will lead to a more concise and informative persistence diagrams. 

We will first define the filtration value for the vertices.
For every vertex, and a value $t>0$ we define
$N_t(v)$ to be the number of vertices in the connected component of $\cG^*_t$ that contains $v$. Fix $k\ge 1$, and define
\eqb\label{eqn:def_tau_v}
\tau_k(v) := \inf \set{t : N_t(v) \ge k}.
\eqe
The edges values are then
\eqb\label{eqn:def_tau_e}
\tau_k((u,v)) = \max(\tau_k(u),\tau_k(v), W((u,v))).
\eqe
Denoting the corresponding filtration by $\cG^{(k)}_t$, note that $\cG^{(1)}_t \equiv \cG^*_t$. 
In other words, compared to $\cG^*_t$, in $\cG_t^{(k)}$ we delay the vertices appearance, until the first time each vertex is contained in a component with at least $k$ vertices (and adjust the edge appearance to be compatible).
Effectively,  the assignment of the new filtration values to the vertices introduces two changes to the persistence diagrams:
\begin{enumerate}
    \item All the points that are linked to components of size smaller than $k$ are removed.
    \item Each birth time corresponds to an edge merging two components $C_1,C_2$ in $\cG_t^*$, such that $|C_1|,|C_2| < k$, and $|C_1|+|C_2|\ge k$.
    \item Each death time corresponds to an edge merging two components larger than $k$.
\end{enumerate}
We call this filtration the `$k$-cluster filtration', to represent the fact that it tracks the formation and merging of clusters of size at least $k$. The parameter $k$ determines what we consider as a sufficiently meaningful cluster. In $\cG_t^*$, every vertex is considered a  cluster, but statistically speaking, this is an overkill. The chosen value of $k$ should depend on the application as well as the sample size.

We conclude this section showing that the $k$-cluster filtrations are decreasing (in a set sense) as we increase $k$. This can be useful, for example, in the context of multi-parameter persistence, which we briefly mention but leave for future work.

\begin{lem} The filtrations $\cG_t^{(k)}$ are decreasing in $k$, i.e., 
$$\tau_{k-1}(x) \leq \tau_k(x),\quad \forall x\in V\cup E.$$

\end{lem}
\begin{proof}
For any vertex $v\in V$, if $|N_t(v)|\geq k$, then $|N_t(v)|\geq k-1$. From \eqref{eqn:def_tau_v} we therefore have that $\tau_{k-1}(v) \le \tau_k(v)$. Using \eqref{eqn:def_tau_e}, we have $\tau_{k-1}(e) \le \tau_k(e)$ for all $e\in E$.
\end{proof}

\section{Algorithm}\label{sec:alg}

In this section, we describe an efficient one-pass algorithm for computing the filtration and persistence diagram at the same time. The time complexity of the algorithm is $O(|E|\times \alpha(|V|))$, where $\alpha(\cdot)$ is the inverse Ackermann function \cite{cormen_introduction_2022}. This is the same complexity as computing the 0-dimensional persistence diagram if we were given the filtration as input. 

We begin with the (standard) terminology and data structures. For simplicity of the description, we assume that the weights on the edges are unique and the vertices have a lexicographical order. We first define a total order on the vertices as follows: the filtration function determines the ordering. Undefined filtration functions are assumed to be $\infty$. If the function is the same or undefined for both vertices, the order is then determined by lexicographical ordering. It is straightforward to check this is a total ordering. 

\begin{rem}
In the case of a total ordering, one can choose a representative of 0-dimensional persistent homology classes -- notably, in the total ordering a unique vertex is the earliest generator for the homology class (i.e., the cluster) which we denote as the canonical representative of the persistent component. 
\end{rem}

To track components as we proceed incrementally through the filtration, we use the union-find data structure, which supports two operations: 
\begin{itemize}
    \item $\text{ROOT}(v)$: returns the canonical representative for the connected component containing $v$.
    \item $\text{MERGE}(u,v)$: merges the connected components containing $u$ and $v$ into one component -- including updating the root.
\end{itemize}
We augment the data structure by keeping track of two additional records:
\begin{itemize}
    \item $\text{SIZE}(v)$: returns the size of the connected component containing $v$.
    \item $\text{COMPONENT}(v)$: returns the list of vertices  in the same component as $v$. 
\end{itemize}
To track the size of the component, we store the size at the root (i.e., the canonical representative) of each component, updating each time a merge occurs. 
To access a connected component, recall that the union-find data structure is implemented as a rooted tree. For each vertex, we store a list of children in the tree. To recover the list of vertices in the component, we perform a depth-first search of the tree starting from the root (although any other traversal method could be used). 
All update operations have $O(1)$ cost (cf., \cite{cormen_introduction_2022}).

Note that when $k=1$, the filtration value for all vertices is $0$ and so the problem reduces to finding the minimum spanning tree of a weighted graph.
Hence, we will assume that $k>1$. Initially, we set the filtration function $\tau(v)=0$ for all vertices, and $\tau(e) = W(e)$ for all edges, and assume the edges are sorted by increasing weight. Note that if this is not the case, this step will be the bottleneck, with a cost of $O(|E|\log |E|)$. Thus, we begin with a forest where each component is a single vertex, i.e. all components are initially born at 0. 

We proceed as in the case of standard  0-dimensional persistence, adding edges incrementally. As no components are added, we are only concerned with merges, the problem is reduced to updating the birth times as we proceed by keeping track of ``active" components (i.e., larger than $k$). We omit points in the persistence diagram which are on the diagonal (birth=death), but these can be included with some additional book-keeping. 

Assume we are adding the edge $e = (u,v)$. If $e$ is internal to a connected component (i.e., $\text{ROOT}(u) = \text{ROOT}(v)$), then it does not affect the $0$-persistence. Otherwise, it connects two components denoted $C_u, C_v$. There are a few cases to consider:
\begin{enumerate}
\item $|C_u\cup C_v| < k$: The merged component is too small to affect the persistence diagram.
We only perform a merge of the components.
\item $|C_u\cup C_v| \geq k$ and  $|C_u| < k$: In this case, $C_u$ becomes active. Thus, we merge the components, and update the value of $\tau$ for all vertices in $C_u$.
$$\tau(x) \leftarrow W(e) \quad\quad \forall x\in C_u $$
is performed. We take similar action if $|C_v|<k$ (or both are less than $k$).
\item $|C_u|,|C_v| \geq k$: Both components are already active and so a new point $(\birth,\death)$ is added to the persistence diagram, with 
\[
\splitb
\birth &= \max\set{\tau(\text{ROOT}(u)), \tau(\text{ROOT}(v))},\\
\death &= W(e).
\splite 
\]
The components are again merged. We note that for any $v$,
$$ \tau(\text{ROOT}(v)) = \min_{x\in C_v} \tau(x).$$
\end{enumerate}
The full procedure is given in Algorithm \ref{alg:main}. Note that we only compute the filtration for the vertices, as the correct edge values can then be computed by Equation \ref{eqn:def_tau_e}.
\begin{algorithm}
\caption{\label{alg:main}One-pass Algorithm}
\begin{algorithmic}[1]
\State $G=(V,E,W)$
\State $\tau: V \rightarrow [0,\infty) $
\State Initialize union-find data structure: $\text{ROOT}(v)=v$ for all $v\in V$
\State $\Dgm,\text{MST}=\emptyset$
\For{$e=(u,v)\in E$} 
     \If{$\text{ROOT}(u)\neq \text{ROOT}(v)$}
        \State $\text{MST}\leftarrow \text{MST}\cup e$
        \If{$\text{SIZE}(u)+\text{SIZE}(v)>k$}
        \If{$\text{SIZE}(u)<k$}
            \For{$x\in \text{COMPONENT}(u)$}
                \State $\tau(x)\leftarrow W(e)$ 
            \EndFor
        \EndIf
        \If{$\text{SIZE}(v)<k$}
            \For{$x\in \text{COMPONENT}(v)$}
                \State $\tau(x)\leftarrow W(e)$
            \EndFor
        \EndIf
        \If{$\text{SIZE}(u),\text{SIZE}(v) \ge k$}
            \State $\birth = \max\{\tau(\text{ROOT}(u)),\tau(\text{ROOT}(v))\}$
            \State $\death = W(e)$
            \State $\Dgm \leftarrow \Dgm \cup (\text{birth},\death) $ 
        \EndIf
        \EndIf
        \State $\text{MERGE}(u,v)$
     \EndIf
 \EndFor \\
 \Return $\Dgm, \text{MST}, \tau$
\end{algorithmic}
\end{algorithm}

\paragraph{Proof of Correctness.} 
We first argue that the function $\tau$ is correctly computed. This follows directly from the fact that  the algorithm explicitly tests when the component contains at least $k$ vertices. The fact that the persistence diagram is correctly computed is a consequence of the following result.
\begin{lem}
The minimum spanning tree for $k=1$ is a minimum spanning tree for any $k$.
\end{lem}
\begin{proof}
The key observation is that until a component contains $k$ vertices, any spanning tree is a minimum spanning tree, as all the edges will be assigned the value when the component becomes active. The remaining edges do not have their values changed and so remain in the MST. 
\end{proof}
The equivalence of the MST and the persistence diagram~\cite{skraba_randomly_2017} then implies correctness of the algorithm.

\paragraph{Proof of Running Time.}
The analysis of the merging is covered verbatim from the standard analysis of the union-find data structure. As described above, the update to the size of the component and updating the list of children in the merge are $O(1)$ operations. All that remains is to prove  is  the cost of updating the function $\tau$. We observe that each vertex is only updated once. This therefore has a total cost of $O(|V|)$, and the edges can be updated at a cost of $O(1)$ per edge (however, there is no practical need for that). This implies the overall running time is $O(|E|\times \alpha(|V))$.

\paragraph{Extracting the Clusters.} To obtain clusters, we can use the algorithm in \cite{chazal_persistence-based_2013}. This algorithm extracts the $\ell$-most persistent clusters by performing merges only when the resulting persistence is less than a threshold. This threshold can be chosen such that there are only $\ell$ points above the threshold in the  diagram. Finally, we note that the cluster extraction can be done on the MST rather than the full graph.

\section{Experiments and Applications}\label{sec:exp}
\subsection{Simulated point-clouds}

We start by generating point-clouds from a mixture of Gaussians, resulting in several blobs of points (Figure~\ref{fig:example_points}). 
We first show the effect of the parameter $k$ on the filtration function and the corresponding persistence diagrams. For the two point-clouds in Figure~\ref{fig:example_points}, we show the resulting persistence diagrams for the $k$-cluster filtrations in Figure~\ref{fig:k-comp-rel}. Notice that the correct number of persistent clusters is evident, especially for $k=10,20,$ and $50$. 
 An important phenomenon that is evident in the figures is that higher values of $k$ filter out more of the `noise'. 

\begin{figure}
 \centering
  \includegraphics[width=0.75\textwidth]{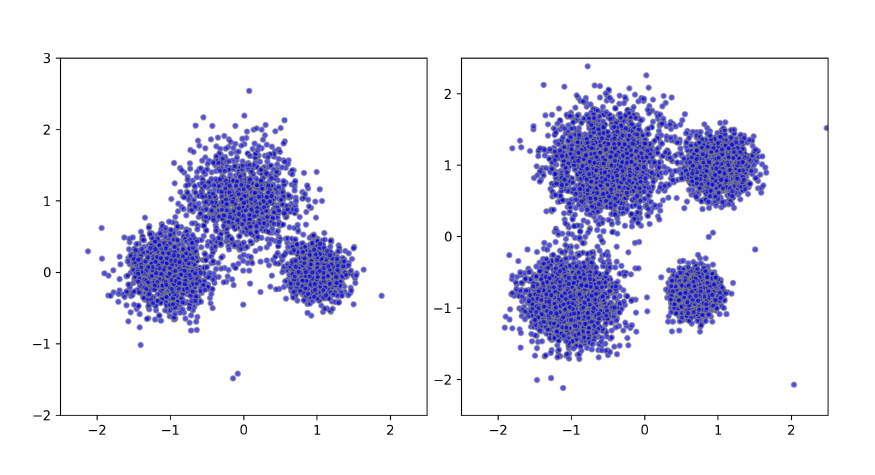}
    \caption{\label{fig:example_points}
    Two examples point-clouds consisting of an i.i.d. sampling from  a mixture of three and four Gaussians  and consisting of 1000 and 2000 points respectively.}
\end{figure}

\begin{figure}
 \centering
  \includegraphics[width=0.99\textwidth]{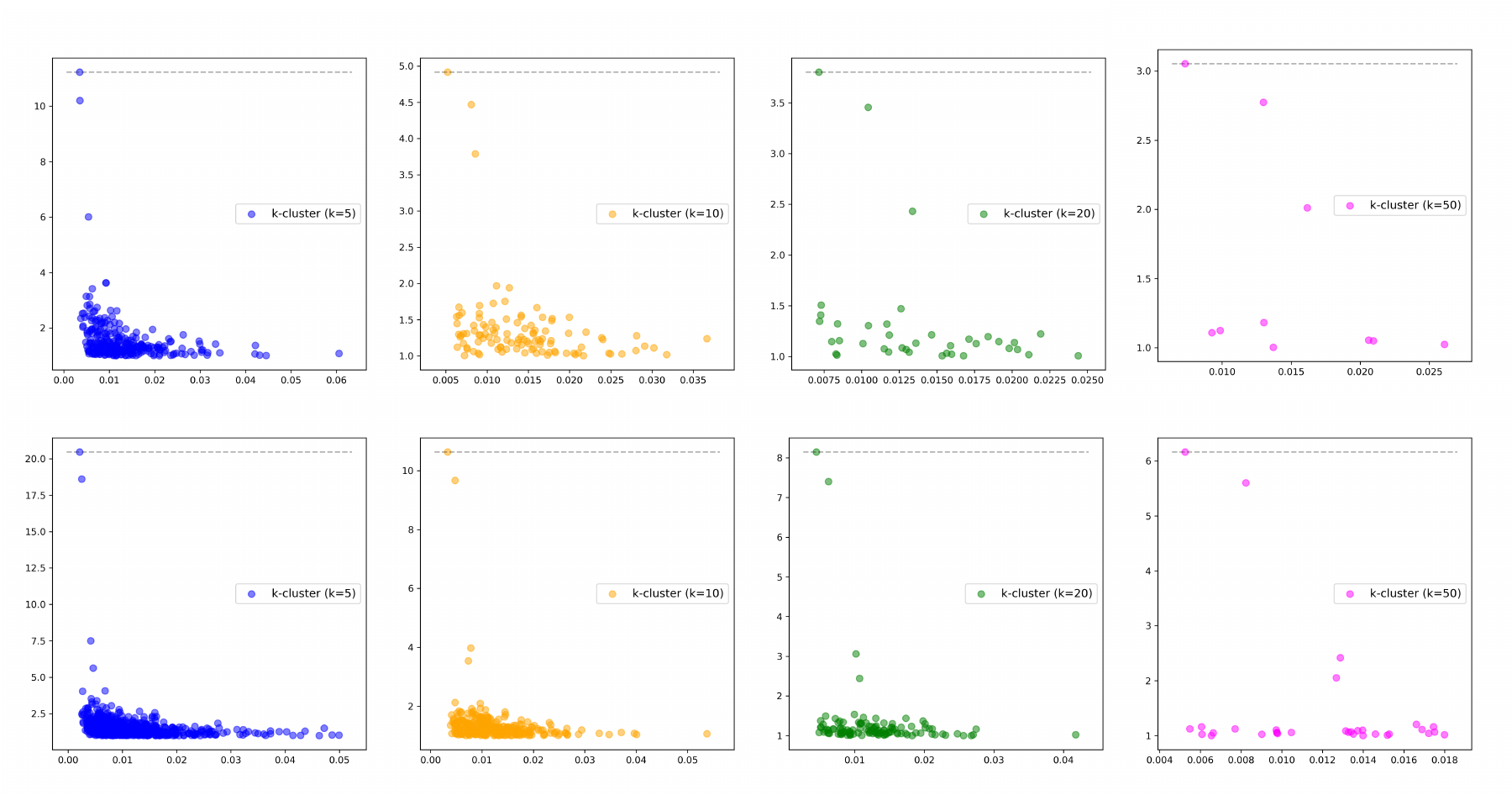}
    \caption{\label{fig:k-comp-rel} The persistence diagrams with death/birth on the $y$-axis with different choices of $k$ for the points sampled from the two mixtures of Gaussians (top row) 3 blobs (bottom row) 4 blobs. Note that the number of outstanding features in the diagrams correspond to the number of clusters in the data.}
\end{figure}

\begin{figure}
 \centering
  \includegraphics[width=0.90\textwidth]{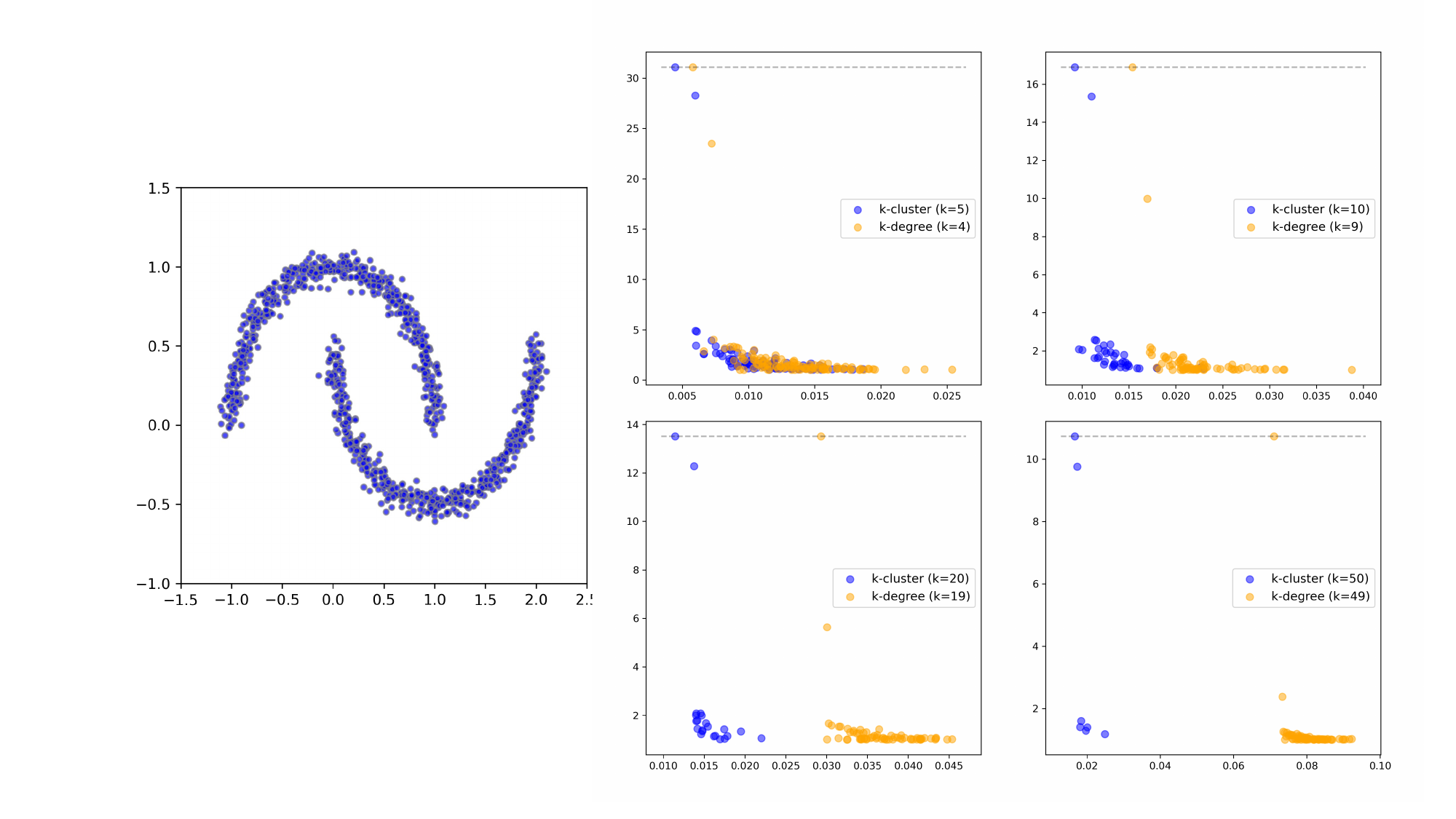}
    \caption{\label{fig:twomoons} A comparison of the $k$-cluster and $k$-degree for the two moons data set. On the right we have the death/birth ratios for different values of $k$.  }
\end{figure}


To place the behaviour of the persistence diagrams into further context, we compare the $k$-cluster filtration with a related construction from the applied topology literature, which has been suggested for dealing with outliers in clustering (and in higher homological dimensions) -- the $k$-degree Vietoris-Rips filtration~\cite{lesnick_interactive_2015}. Given a weighted graph $G = (V,E,W)$, we define the $k$-degree filtration, denoted $\delta_k:(V\cup E)\to [0,\infty)$ as follows. For every vertex $v\in V$ we take $\delta_k(v)$ to be its $k$-nearest neighbor distance. The values of the edges, is then determined the same as in \eqref{eqn:def_tau_e}. The $k$-degree filtration has been used in the context of multi-parameter persistence, with the bifiltration induced by decreasing $k$ and increasing the edge weight (commonly, Euclidean distance).
In this paper, we do not explore the multi-parameter setting. Rather, we focus the properties of the persistence diagrams for a fixed $k$. 
We make two observations before investigating the differences:
\begin{enumerate}
\item The $k$-degree filtration function is determined completely by the local neighborhood of a vertex (i.e., its immediate neighbors in the graph). The same is not true for the $k$-cluster filtration.
\item For a fixed value of $k$ we have $\tau_k(v)\le \delta_{k-1}(v) $ for all $v\in V$. In other words, the value of $k$-cluster function is less than or equal to than the value of the $(k-1)$-degree function. This follows from the fact that if a vertex has $k-1$ neighbors, then it is part of a cluster of at least $k$ vertices. 
\end{enumerate}

In Figure \ref{fig:twomoons}, we show the relative persistence diagrams for two non-convex clusters for both the $k$-degree and $k$-cluster filtrations, for different values of $k$. In this example, especially for larger $k$, the  persistent clusters are much more prominent in the $k$-cluster filtration compared to the $k$-degree filtration. This  may be explained by the fact that a much larger radius is needed to obtain the required number of neighbors.  In Figure~\ref{fig:deg_vs_cluster}, we show the same comparison for relative persistence diagrams for 3 and 4 blobs, where the difference between the two methods is less clear. However,  Figure \ref{fig:kstability} highlights an additional  difference in the behaviors of the two filtrations. In this figure, we compare the persistence (death/birth) for the second most persistent cluster, for a wide range of $k$ values. In the left and center plots, the second most persistent cluster corresponds to a true cluster in the data. We observe that the persistence value decays much more slowly for the $k$-cluster filtration, i.e. the true cluster remains more persistent for increasing values of $k$. The plot on the right presents the same comparison, but for uniformly distributed random points. In this case, the second most persistent cluster is by construction noise (i.e., not a real cluster in the data). Here although the $k$-cluster filtration decays more slowly, it is comparable to the $k$-filtration. Hence we can conclude that persistent clusters show a more stable behavior over ranges of $k$ for the $k$-cluster filtration compared to the $k$-degree filtration.

\begin{figure}
 \centering
  \includegraphics[width=0.99\textwidth]{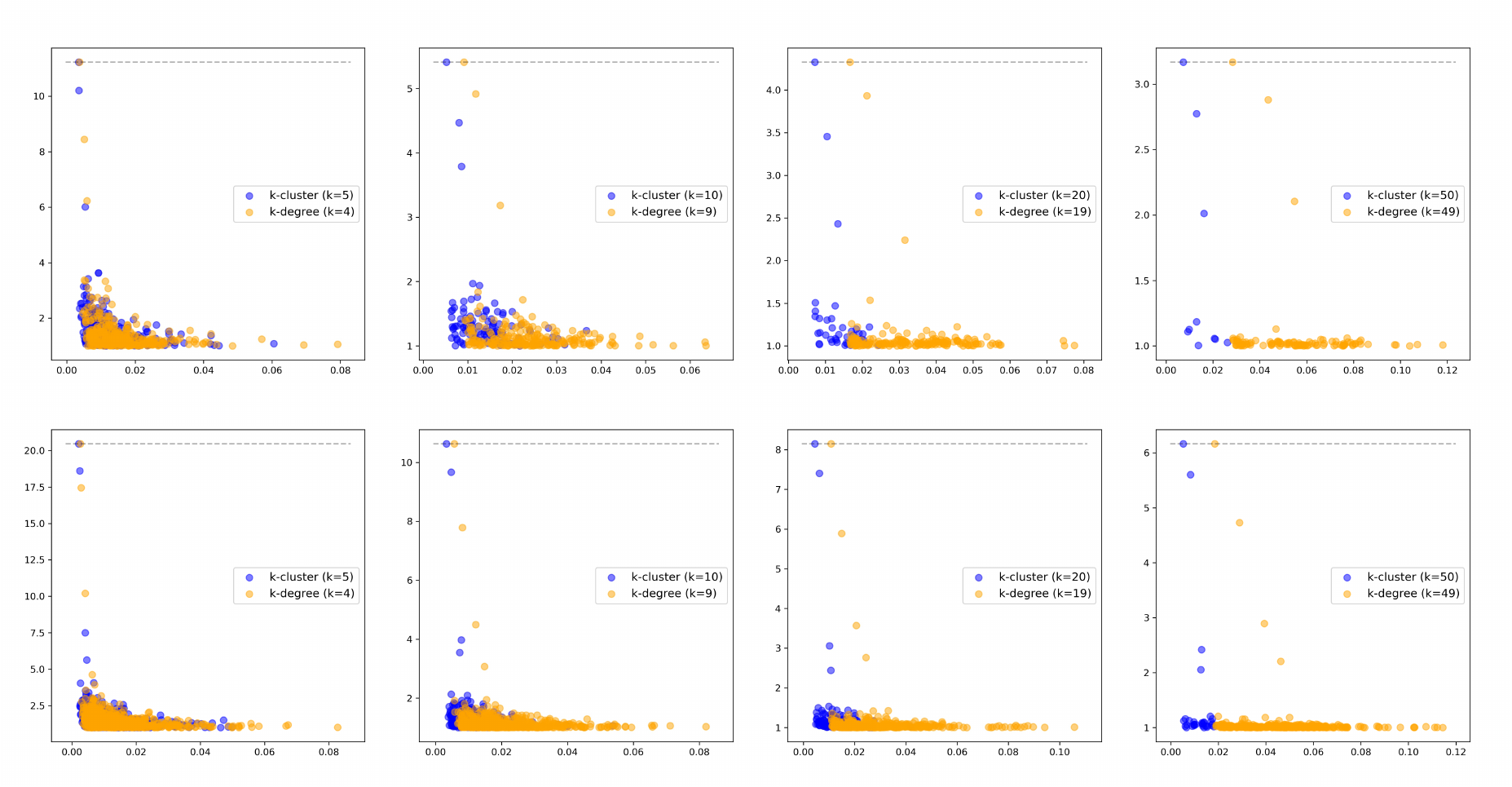}
    \caption{\label{fig:deg_vs_cluster} (top row) 3 blobs, (bottom row) 4 blobs.  The relative persistence diagrams for each point cloud, with the $k$-degree filtration in yellow and the $k$-cluster filtration in blue for $k=5,10,20,$ and $50$.  }
\end{figure}

\begin{figure}
 \centering
  \includegraphics[width=0.90\textwidth]{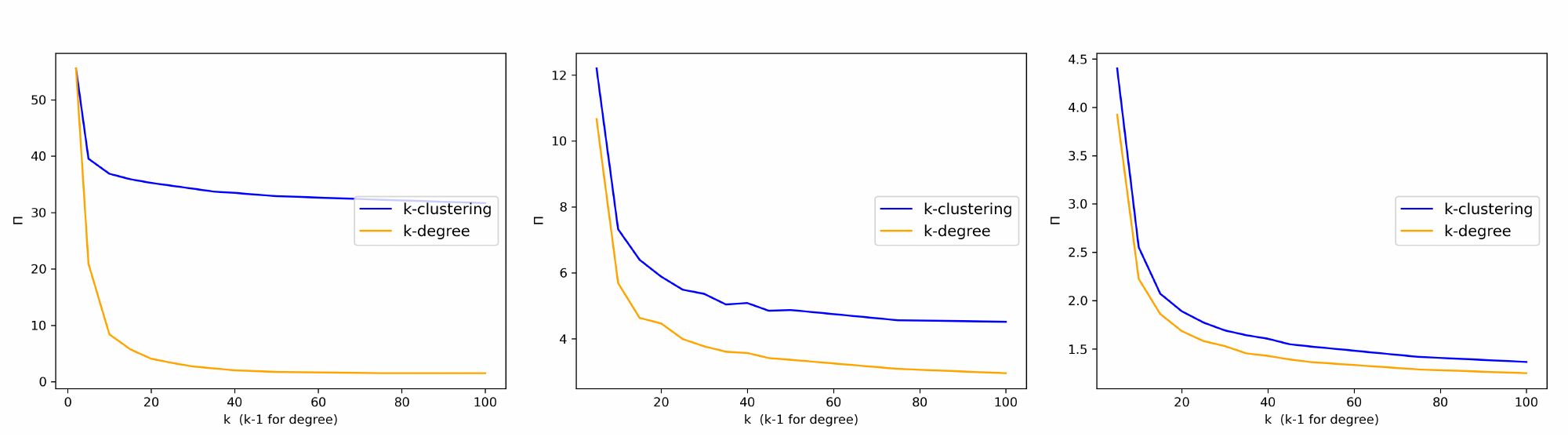}
    \caption{\label{fig:kstability} The effect on the second most persistent cluster for different values of $k$. On the left and center, this corresponds to a true cluster (left -- two moons and center -- mixture of 3 Gaussians). On the right --uniform random points. Here the noise cluster drops nearly as quickly in both cases. }
\end{figure}

\subsection{Universality}\label{sec:univ}

In \cite{bobrowski_universal_2023}, we published a comprehensive experimental work, showing that the distribution of persistence values is universal. We consider a persistence diagram as a finite  collection of points in $\R^2$, $\dgm = \set{(b_1,d_1),\ldots, (b_M,d_M)}$. For each point $p_i = (b_i,d_i)$ we consider the multiplicative persistence value $\pi(p_i) = d_i/b_i$. Our goal is to study the distribution of the $\pi$-values across an entire diagram.

Our results in \cite{bobrowski_universal_2023} are divided into two main parts. Given a point cloud of size $n$, we compute the persistence diagram for either the \v Cech or the Vietoris-Rips filtrations.
In \emph{weak universality} we consider the empirical measure
\[
\Pi_{n} := \frac{1}{|\dgm_k|}\sum_{p\in\dgm_k}\delta_{\pi(p)},
\]
and we conjecture that for iid samples, we have
\[
\limninf \Pi_n = \Pi^*_{d,\cT,k},
\]
where $d$ is the dimension of the point-cloud, $k$ is the degree of homology, and $\cT$ is the filtration type (i.e., \v Cech or Vietoris-Rips). In other words, the limiting distribution for the $\pi$-values depends on $d,k,\cT$ but is  independent of probability distribution generating the point-cloud.

In \emph{strong universality} we present a much more powerful and surprising conjecture. Here, we define $\ell(p) := A\logg(\pi(p)) + B$ (the values of $A$ and $B$ are speficied in \cite{bobrowski_universal_2023}), and the empirical measure
\[
\cL_n := \frac{1}{|\mathrm{dgm_k}|}\sum_{p\in\dgm_k}\delta_{\ell(p)}.
\]
Our conjecture is that for wide class of random point-clouds (including non-iid and real-data), we have 
\[
\limninf \cL = \cL^*,
\]
where $\cL^*$ is a unique universal limit. Furthermore, we conjecture that $\cL^*$ might be the left-skewed Gumbel distribution. 

Originally, the results in \cite{bobrowski_universal_2023} are irrelevant for the $0$-th persistence diagram of random point-clouds, as the birth times are all zeros. However, once we replace the standard filtration with the $k$-cluster filtration, we have a new persistence diagrams with non-trivial birth time that we can study.  In Figure \ref{fig:universal} we demonstrate both weak and strong universality properties for the $k$-cluster persistent homology. We generated iid point-clouds across different dimensions, with different distributions (uniform in a box, exponential, normal). The results show that both weak and strong universality hold in these cases as well. We note that for weak universality, the limiting distribution depends on both $d$ (dimension of point-cloud) and $k$ (minimum cluster size).

\begin{figure}
    \centering
    \includegraphics[width=0.99\textwidth]{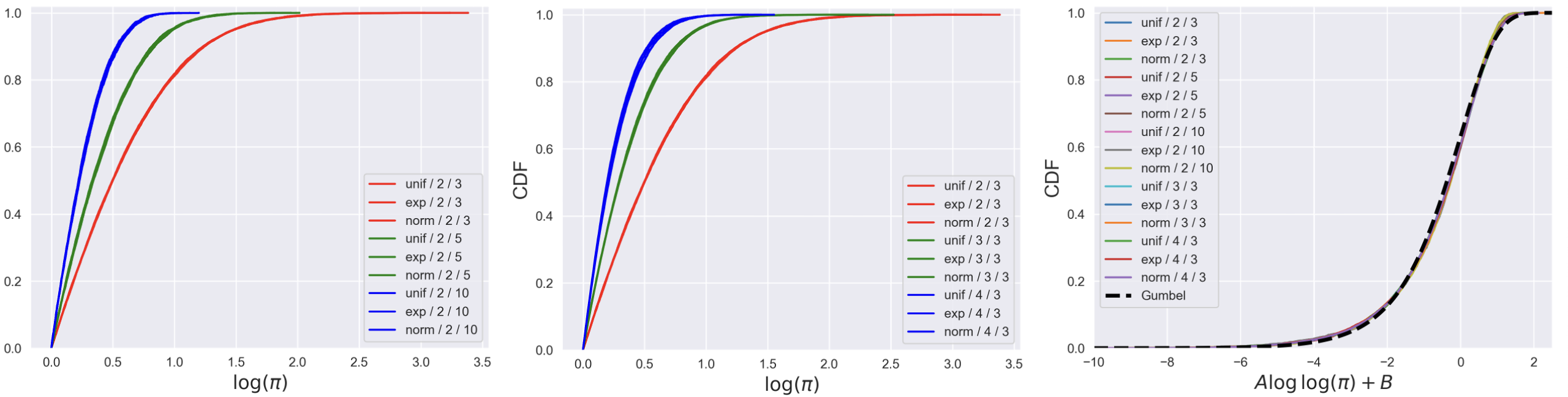}
    \caption{\label{fig:universal}Universal distribution for  $k$-cluster persistence. The labels in the legend are structured as distribution/$d$/$k$, where $d$ is the point-cloud dimension, and $k$ is the cluster size. The distributions taken are uniform in a unit box, exponential, and normal. The first two plots show that weak universality holds, and that the limit depends on $d,k$, but not on the distribution. The rightmost plot, demonstrates that strong universality holds under a proper normalization. We also included the left-skewed Gumbel distribution (dashed line) for comparison. }
    \label{fig:enter-label}
\end{figure}

\subsection{Clustering}

As mentioned in the introduction, a key motivation for this work was to apply the $k$-cluster filtration to clustering. To obtain a clustering from a 0-dimensional persistence diagram, we use the algorithm proposed in \cite{chazal_persistence-based_2013}. Roughly speaking, given a threshold $\alpha$, it extracts all clusters which are more than $\alpha$-persistent. We note that the original measure  for persistence in \cite{chazal_persistence-based_2013} was given by $d-b$, however the change to use $d/{b}$ in the algorithm is trivial. 

\paragraph{Statistical Testing.} An important consequence of the universality results in Section \ref{sec:univ} is that  the limiting distribution (after normalization) appears to be a known distribution, i.e. left-skewed Gumbel. We can thus perform statistical testing on the number of clusters as in \cite{bobrowski_universal_2023}. The null-hypothesis denoted by  $\mathcal{H}_0^{(i)}$, is that the $i$-th most persistent cluster is due to noise. Assuming the universality conjectures hold, the null hypothesis is given in terms of the $\ell$-values as
$$
\mathcal{H}_0^{(i)}\ :\ \ell(p_i) \sim \mathrm{LGumbel}.
$$
where $p_i$ represents the $i$-th most persistent cluster in terms of death/birth. The corresponding p-value is given by
\begin{equation*}
\text{p-value}_i = \cprob{\ell(p_i) \ge x}{\mathcal{H}_0^{(i)}} = e^{-e^x}.
\end{equation*}
Note that since we are testing sorted values, we must use a multiple hypothesis testing correction. In the experiments we describe below, we use the Bonferroni correction.

\begin{figure}
 \centering
 \vspace{-1cm}
  \includegraphics[height=0.95\textheight]{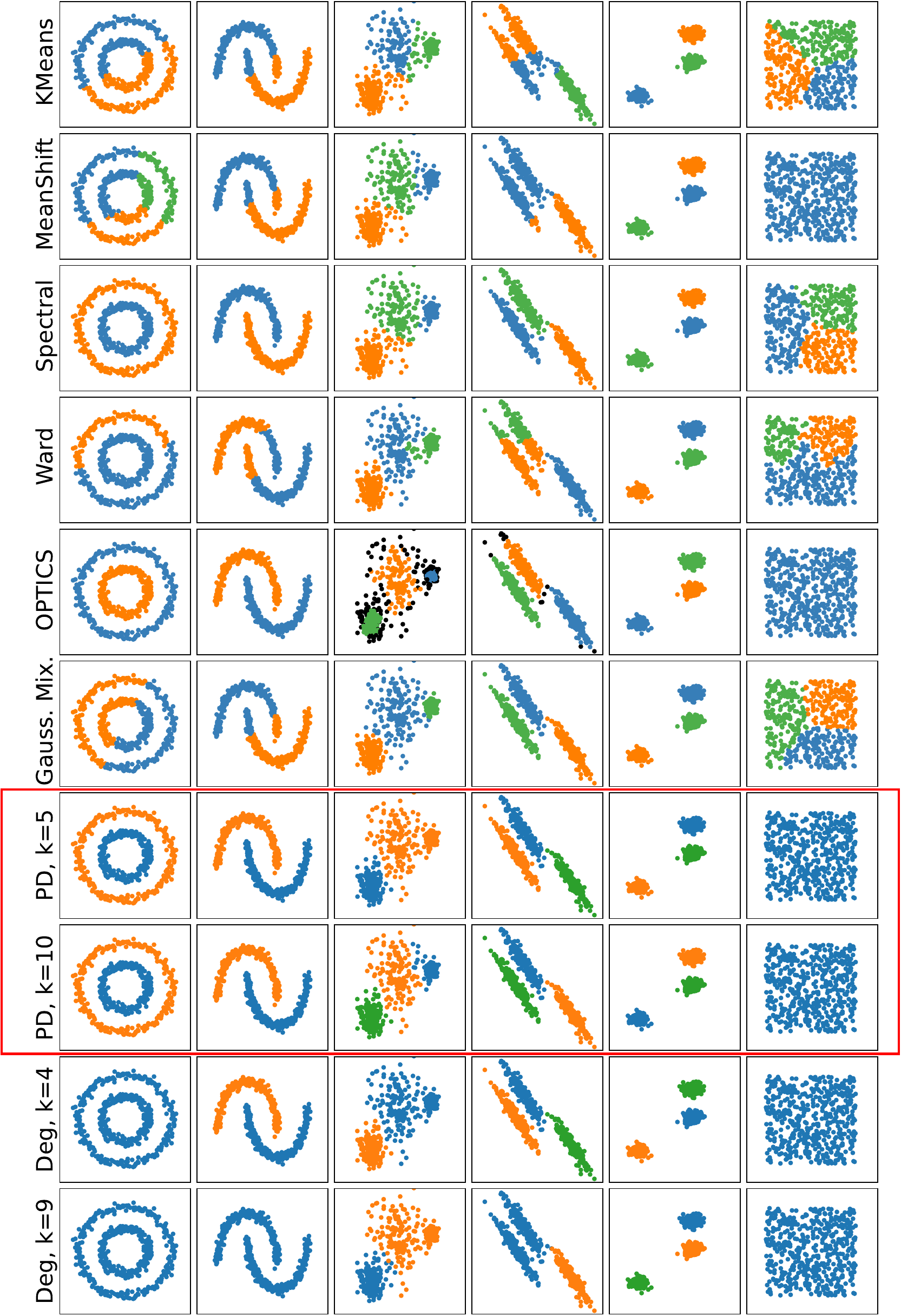}
    \caption{\label{fig:clustering} A comparison of standard clustering examples for different clustering approaches. In the case of the $k$-cluster filtration (PD) and $k$-degree filtration (Deg), the number of clusters was chosen using statistical significance testing.  }
\end{figure}

In Figure \ref{fig:clustering}, we compared the $k$-cluster filtration and the $k$-degree filtration using persistence based clustering from \cite{chazal_persistence-based_2013} with other common algorithms for clustering. For the other approaches, we used the standard implementations found in \cite{pedregosa_scikit-learn_2011}, which have associated techniques for choosing the number of clusters.  
In the cases of the $k$-cluster filtration and the $k$-degree filtration, the number of clusters was chosen using the statistical testing described above. Note that since the number of points in the standard examples was quite small, we limited $k$ to $5$ and $10$. The best result is for the  $k$-cluster filtration with $k=10$ ($k=5$ fails to identify one of the clusters in the third example). The $k$-degree filtration performs well but the additional ``noise" points in the diagram, mean that some clusters are not identified as significant. 

\paragraph{Clustering on Trees.} As a second example, we describe clustering on weighted trees. 
We generated a uniform random tree on $n$ vertices, and assigned uniformly distributed random weights on the edges (between $0$ and $1$). We show an example in Figure \ref{fig:tree}.
The methods seems to capture certain structure about the tree, although we leave further investigation of this structure as future work. 

Note that in the tree case, it is often impossible to use $k$-degree filtrations, as the tree will have vertices with degree smaller than $k$ that will never be included in the filtration, whereas for the $k$-clustering filtration, all nodes are included as long as the underlying graph is connected  (or all components have at least $k$ vertices). We note that it is possible to use an alternative definition for the $k$-degree filtrations, by embedding the tree into a metric space (i.e., using the graph metric induced by the weights). However, this is similar to studying a complete graph induced by the metric which is somewhat different than studying the graph directly. We use this method in the rightmost plot of Figure \ref{fig:tree}.


\begin{figure}
 \centering
  \includegraphics[width=0.99\textwidth]{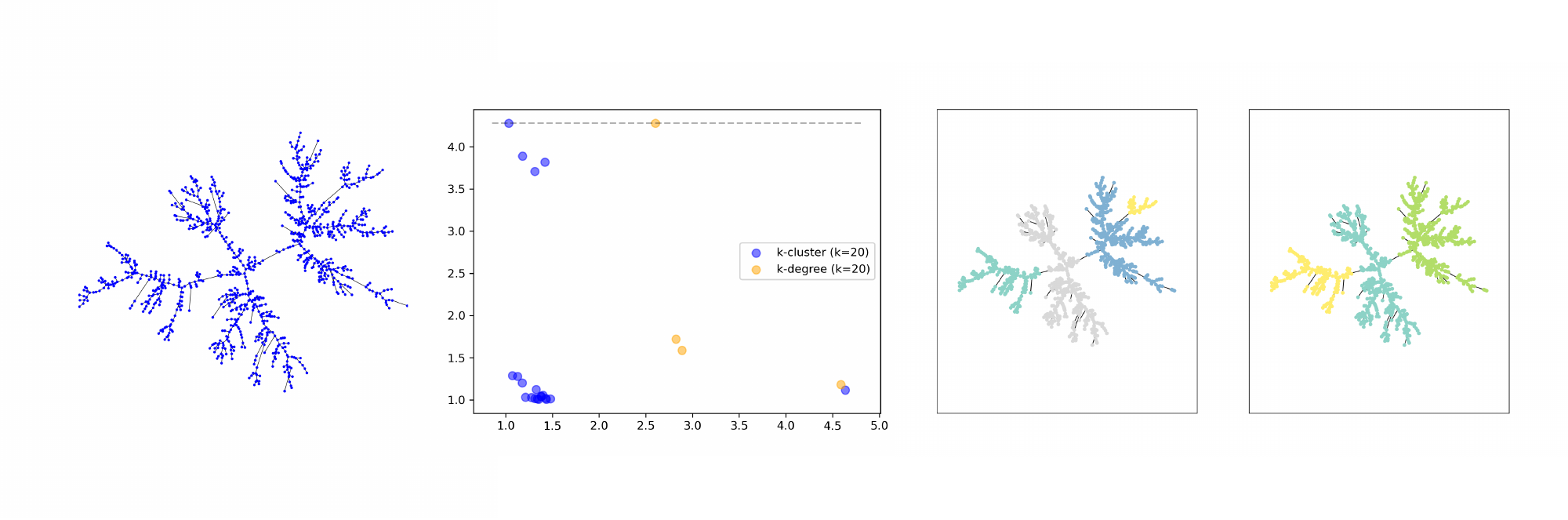}
    \caption{\label{fig:tree} A clustering on a uniform random tree. The threshold with $k$-clustering gives 4 clusters while only 3 with the (metric) $k$-degree. }
\end{figure}

\section{Probabilistic Analysis}
\label{sec:prob}

In this section we wish to revisit some of the fundamental results known for the (persistent) homology of random graphs and simplicial complexes, and show that analogous statements hold for our new $k$-cluster filtration. We provide here the main statements. Proofs are available in the appendix.

\subsection{Connectivity}
We will consider two models here. In the $G(n,p)$ random graph we have $n$ vertices, and each edge is placed independently with probability $p$. In the $G(n,r)$ random geometric graph, we take 
a homogeneous Poisson process $\cP_n$ on the $d$-dimensional flat torus, with rate $n$. Edges are then placed between vertices that are less than distance $r$ apart.
In both models, connectivity results are tied to the expected degree. For the $G(n,p)$ model we define $\Lambda =np$, and for the $G(n,r)$ we take $\Lambda = n\omega_dr^d$. Then in \cite{erdos_random_1959} and \cite{penrose_longest_1997} the following was proved.
\begin{thm}\label{thm:rg_conn}
Let $G_n$ be either $G(n,p)$ or $G(n,r)$. Then
\[
\limninf\prob{G_n\text{ is connected}} = \begin{cases} 1 & \Lambda = \log n + w(n),\\
0 & \Lambda = \log n - w(n).\end{cases}
\]
\end{thm}
A key element in proving connectivity (for either models) is to show that around $\Lambda = \log n$, the random graph consists of a single giant component, a few isolated vertices, and nothing else. Thus, connectivity is achieved when the last isolated vertices gets connected.

Our goal in this section is to analyze connectivity in the $G(n,p)$ and $G(n,r)$ model, via our new $k$-cluster filtration. 
Note that for a fixed $n$, we can view both models as filtrations over the complete graph. For the $G(n,p)$ model the weights of the edges, are independent random variables, uniformly distributed in $[0,1]$. For the $G(n,r)$ the weight of an edge is given by the distance between the corresponding points in the torus. We define $G^{(k)}(n,p)$ and $G^{(k)}(n,r)$ to be the random filtrations generated by changing the filtration function to be $\tau_k$. 
Our goal here is to explore the phase transition for the $k$-cluster connectivity. As opposed to connectivity in the original random graphs, the results here differ between the models.
\begin{thm}\label{thm:pt_er}
For the $G^{(k)}(n,p)$ filtered graph we have,
    \[
        \limninf\prob{G^{(k)}(n,p)\text{ is connected}} = \begin{cases}1 & \Lambda = \frac1k(\log n + (k-1)\logg n) +w(n),\\
        0 & \Lambda = \frac1k(\log n + (k-1)\logg n) - w(n),\end{cases}
    \]
    for any $w(n) = o(\logg n)$ such that $w(n)\to\infty$.
\end{thm}

For the $G^{(k)}(n,r)$ model, proving the connectivity is a much more challenging task and beyond the scope of this paper. 
The following statement, however, is relatively straightforward to prove.
\begin{prop}\label{prop:N_k}
Let $N_k = N_k(n,r)$ be the number of connected components of size $k$ in $G(n,r)$. Then,
     \[
        \limninf\prob{N_k =0} = \begin{cases}1 & \Lambda = \log n-(d-1)(k-1)\logg n+w(n),\\
        0 & \Lambda = \log n - (d-1)(k-1)\logg n - w(n).\end{cases}
    \]
    for any $w(n) = o(\logg n)$ such that $w(n)\to\infty$.
\end{prop}
From this lemma we conclude that when $\Lambda = \log n -(d-1)(k-1)\logg n -w(n)$  the graph $G(n,r)$ has components of size $k$, which implies that $G^{(k)}(n,r)$ is not connected.
On the other hand, when $\Lambda = \log n -(d-1)(k-1)\logg n +w(n)$, we have $N_j=0$ for all fixed $j\ge k$. Which indicates that $G^{(k)}(n,r)$ should be connected. This leads to the following conjecture.
\begin{con}
For the $G^{(k)}(n,r)$ filtered graph we have,
    \[
        \limninf\prob{G^{(k)}(n,r)\text{ is connected}} = \begin{cases}1 & \Lambda = \log n-(d-1)(k-1)\logg n+w(n),\\
        0 & \Lambda = \log n - (d-1)(k-1)\logg n - w(n).\end{cases}
    \]
\end{con}
Note that both phase transitions occur before the ones for the original graph models. This is due to the fact that for $k>1$ the $k$-cluster filtration does not allow having any isolated vertices. Also note that taking $k=1$ both results coincide with Theorem \ref{thm:rg_conn}. 

\subsection{Limiting Persistence Diagrams}

In \cite{hiraoka_limit_2018}, it was shown that for stationary point processes, persistence diagrams have a non-random limit (in the vague convergences of measures). A similar statement will hold for the $k$-cluster persistence diagrams.

Let $\Dgm^{(k)}(\cP)$ be the $k$-cluster persistence diagram for a point-cloud $\cP$. We define the discrete measure on $\R^2$,
\[
\xi^{(k)}(\cP) := \sum_{(b,d)\in\Dgm^{(k)}(\cP)} \delta_{(b,d)}.
\]
Let $Q_L = [-L/2,L/2]^d$. The following is an analogue of Theorem 1.5 in \cite{hiraoka_limit_2018}.
\begin{thm}\label{thm:limit_pd}
Assume that $\cP$ is a stationary point process in $\R^d$ with all finite moments. For any $k$, there exists a deterministic measure $\mu_k$, such that
\[
\lim_{L\to\infty}\frac{1}{L^d}\mean{\xi^{(k)}(\cP\cap Q_L)}  = \mu_k,
\]
where the limit is in the sense of vague convergence. Furthermore, if $\cP$ is ergodic, then almost surely
\[
\lim_{L\to\infty}\frac{1}{L^d}{\xi^{(k)}(\cP\cap Q_L)}  = \mu_k.
\]
\end{thm}

\subsection{Maximal Cycles}

In \cite{bobrowski_maximally_2017} the largest cycles in persistence diagrams were studied. Specifically, for every point $p= (b,d)$ in a diagram, we compute the so-called $\pi$-value - $\pi(p) = d/b$, as a scale-invariant measure of size. Considering the homogeneous Poisson process $\cP_n$, we define $\Pi_{k,\max}$ as the largest $\pi$-value in the $k$-th persistent homology. The main result in \cite{bobrowski_maximally_2017} then states that with high probability
\[
    A_k \Delta_k(n) \le \Pi_{k,\max} \le B_k\Delta_k(n),
\]
where $A_k,B_k>0$ are constants, and 
\[
    \Delta_k(n) = \param{\frac{\log n}{\logg n}}^{1/k}.
\]
For the $k$-cluster persistence, we will show that the largest $\pi$-value has a completely different scaling.
\begin{thm}
    Let $\cP_n$ be a homogeneous Poisson process in the flat torus, with rate $n$. Let $\Pi^{(k)}_{\max}$ denote the maximum $\pi$-value in the $k$-cluster persistence diagram (excluding the infinite cluster).
    Then, for every $\eps>0$ we have
    \[
        \limninf\prob{n^{\frac1{d(k-1)}-\eps} \le \Pi^{(k)}_{\max} \le n^{\frac1{d(k-1)}+\eps}} = 1.
    \]
\end{thm}
\begin{rem}
We observe that the largest $\pi$-value in the $k$-cluster persistence, is significantly larger than that of the $k$-dimensional homology. The main reason for that is the following. In \cite{bobrowski_maximally_2017}, our upper bound for $\Pi_{k,\max}$ used an iso-perimetric inequality, which implies that large $\pi$-values require large connected components. However, the $\pi$-values in the $k$-cluster persistence, only require a cluster of size $k$ to be formed, and thus can be generated by much smaller connected components. 
\end{rem}



\appendix
\section*{Appendix}
In this appendix we provide the proofs for the statements made in Section \ref{sec:prob}.

\section{Connectivity}
\begin{proof}[Proof of Theorem \ref{thm:pt_er}]
  
    Note that for the $k$-cluster filtration, connectivity is equivalent to  the original $G(n,p)$ graph having no components of size $j$ for any $k\le j \le n/2$. Let $N_j = N_j(n,p)$ be the number of components of size $j$ in $G(n,p)$. 
    Taking similar steps to the proof of connectivity for random graphs (e.g., \cite{frieze_introduction_2016}), we have
    \eqb\label{eqn:N_j_bound}
        \mean{N_j} \le \binom{n}{j} j^{j-2}p^{j-1}(1-p)^{j(n-j)}. 
    \eqe
    For $k+1\le j< 4k$, we have 
    \[
        \mean{N_j}\le C n^j \param{\frac{\Lambda}{n}}^{j-1} e^{-j(n-4k)(\Lambda/n)},
    \]
    for some $C>0$. Taking  $\Lambda = \frac1k (\log n + (k-1)\logg n)+c$, we have
    \[
        \mean{N_j} \le C n^{-1/k} (\log n)^{j-1} e^{4j\log n /n }.
    \]
    For $4k \le j \le n/2$, we have
    \[
        \mean{N_j} \le \param{\frac{ne}{j}}^j j^{j-2} \param{\frac{\Lambda}{n}}^{j-1} e^{-j\Lambda/2}\le  \frac{n}{j^2} \param{\frac{e^{1-c/2}\log n}{n^{1/2k}}}^j \le \frac{n}{j^2} n^{-j/3k}.
    \]
    Therefore,
    \[
    \sum_{j=4k}^{n/2} \mean{N_j} \le \frac{1}{8k^2} n^{-1/3}.
    \]    
    To conclude, we showed that 
    \[
        \limninf \sum_{j=k+1}^{n/2} \mean{N_j}  = 0.
    \]
    This implies that for $\Lambda = \frac1k(\log n + (k-1)\logg n)+c$, we have
    \[
        \prob{G^{(k)}(n,p) \text{ is connected}} \approx \prob{N_k > 0}.
    \]
    Similar estimates to the ones above, show that
    \[
    \mean{N_k} \approx e^{-kc}.
    \]
    Therefore, when $c=w(n)\to\infty$, we have $\probx{N_k>0} \to 0$. Together with a second-order argument, we can similarly show that when $c=-w(n)$, we have $\prob{N_k>0} \to 1$.
This concludes the proof. 
\end{proof}

\begin{proof}[Proof of Proposition \ref{prop:N_k}]
Recall that $N_k$ is the number of components of size $k$ in $G(n,r)$. 
In \cite{penrose_k_2022}[Theorem 3.3], it was shown that 
\[
\mean{N_k}\approx \var{N_k} \approx C_kn \Lambda^{-(d-1)(k-1)}e^{-\Lambda},
\]
for some constant $C_k>0$. When $\Lambda = \log n -(d-1)(k-1)\logg n + w(n)$, we have $\mean{N_k} \to 0$, implying that $\prob{N_k>0} \to 0$. When $\Lambda = \log n -(d-1)(k-1)\logg n -w(n)$, we can use Chebyshev's inequality,
\[
\prob{N_k=0} \le \prob{|N_k-\mean{N_k}| \ge \mean{N_k}} \le \frac{\var{N_k}}{(\mean{N_k})^2} \approx \frac1{\mean{N_k}} \to 0.
\]
This completes the proof.
\end{proof}

\section{Maximal \texorpdfstring{$\pi$}{p}-value}

\begin{proof}
Let $r,R$ denote the birth and death radii of a cluster in the $k$-cluster persistence diagram, and recall that $\pi = R/r$.

For an upper bound, we denote by $N_k(r)$ the number of connected subsets of size $k$, at radius $r$. Using Mecke's formula (cf. \cite{penrose_random_2003}),
\[
\splitb
\mean{N_k(r)} &= \frac{n^k}{k!}\int_{(\T^d)^k} \indf{G(\bx,r)\text{ is connected}}d\bx\\
&=\frac{n\lambda^{k-1}}{k!} \int_{(\R^d)^{k-1}}\indf{G((0,\by),1)\text{ is connected}}d\by,\\
&=C_k n\lambda^{k-1}, 
\splite
\]
where $C_k$ is a positive constant, $\lambda = nr^d$, and we used the change of variables $x_i \to x_1+ry_i$ ($i=2,\ldots,k$).
For any $\eps>0$, if $\lambda = n^{-1/(k-1)-\eps}$ then $\mean{N_k(r)}\to 0$. Thus, we can assume that with high probability the birth time of all $k$-clusters has $\lambda \ge n^{-1/(k-1)-\eps}$.
In addition, from Theorem \ref{thm:rg_conn}, if we denote $\Lambda = nR^d$, then when $\Lambda = C\log n$ the graph $G(n,r)$ is connected. This implies that with high probability all death times of $k$-clusters have $\Lambda \le C\log n$.

These bounds together imply that with high probability, for all the points in the $k$-cluster persistence diagram, for any $\eps>0$, we have
\[
\pi = \param{\frac{\Lambda}{\lambda}}^{1/d} \le \param{\frac{C\log n}{n^{-1/(k-1)-\eps}}}^{1/d}.
\]
Therefore, for any $\eps>0$, we have
\[
\limninf\prob{\Pi_{\max}^{(k)} \le n^{\frac{1}{d(k-1)}+\eps}} = 1.
\]

For the lower bound, we denote by $\hat N_k(r,R)$ the number of components of size  $k$, born before $r$, that are isolated at radius $R$ (and hence die after $R$). Then
\[
\splitb
\meanx{\hat N_k(r,R)} &= \frac{n^k}{k!}\int_{(\T^d)^k}\indf{G(\bx,r)\text{ is connected}} e^{-n\vol(B_R(\bx))}d\bx,\\
&=\frac{n\lambda^{k-1}}{k!}\int_{(\R^d)^{k-1}}\indf{G((0,\by),1)\text{ is connected}} e^{-n\vol(B_R(0,r\by))}d\by,\\
\splite
\]
where $B_R(\bx)$ is the union of balls of radius $R$ around $\bx$.
We will apply the dominated convergence theorem, using the fact that when $r/R\to 0$, we have
\[
\limninf \frac{\vol(B_R(0,0+r\by))}{\omega_d R^d} = 1.
\]
This leads to
\[
 {\meanx{\hat N_k(r,R)}}\approx C_k n\lambda^{k-1}e^{-\omega_d\Lambda}.
\]
Taking $\Lambda = C>0$, and $\lambda = n^{-1/(k-1)+\eps}$, we have
\[
\meanx{\hat N_k(r,R)} \to \infty.
\]
Using a second moment argument will show that for all $\eps>0$
\[
\prob{\Pi_{\max}^{(k)} \ge n^{\frac1{d(k-1)} -\eps}} \to 1,
\]
completing the proof.
\end{proof}

\section{Limiting Persistence Diagram}

The key part of the proof in \cite{hiraoka_limit_2018}, is bounding the add-one cost of the \emph{persistent} Betti numbers. Let $G=(V,E,W)$ be a weighted graph, and $\{G^{(k)}_t\}$ be the corresponding $k$-cluster filtration.
Define $\beta_0^{r,s}(G^{(k)})$ as the $0$-th persistent Betti number, i.e., the number of components born in $t\in [0,r]$ and die at $(s,\infty]$ (for a formal definition, see \cite{hiraoka_limit_2018}). Fix an edge $e_0\not\in E$, with a given weight $W(e_0) = w_0$, and let $\tilde G = (V,\tilde E, \tilde W)$ be a weighted graph with $\tilde E = E\cup\{e_0\}$, and 
\[
\tilde W(e) := \begin{cases}
    W(e) & e\ne e_0,\\
    w_0 & e = e_0.
\end{cases}
\]
Let $\{\tilde G^{(k)}_t\}$ denote the corresponding $k$-cluster filtration.
The entire proof Theorem \ref{thm:limit_pd} follows verbatim from the proofs in \cite{hiraoka_limit_2018}, provided that we prove the following lemma.

\begin{lem}
$$\left| \beta_0^{r,s}(\tilde G^{(k)}) - \beta_0^{r,s}(G^{(k)})\right| \leq 1.$$
\end{lem}
In other words, if we add a single edge to the filtration, the number of persistent clusters can change by at most $1$. Note that the proof here is not a straightforward application of Lemma 2.10 in \cite{hiraoka_limit_2018}, since in our case, when a single edge is added to the filtration, the filtration values of other vertices and edges might be affected.

\begin{proof}

Let $e_0=(u,v)$ with $W(e_0)=w_0$. Let $C_u$ and $C_v$ denote the components of the end points of $e_0$ at $w_0$ in the original filtration $\{G_t^{(k)}\}$. There are three possible cases which can occur. 

\noindent {\bf Case I:} Both $|C_u|<k$ and $|C_v| < k$. Note that in this case $\tau_k(u),\tau_k(v) > w_0$. Let $C'_u$ be the cluster of $u$ at $\tau_k(u)$, so that it is the component of $u$ when it first appears in $G^{(k)}_t$. Similarly define $C'_v$. Note that aside from $C'_u\cup C'_v$ the filtration value of all other vertices remains unchanged by adding $e_0$.

Without loss of generality, suppose that $w_0 < \tau_k(u) < \tau_k(v)$. Then comparing the persistence diagram for $G_t^{(k)}$ and $\tilde G_t^{(k)}$, only two differences can occur:
\begin{enumerate}
    \item The point representing $C'_v$ in $G_t^{(k)}$ is removed, since $C'_v$ is no longer a connected component in $\tilde G_t^{(k)}$ (as it is merged with $C'_u$).
    \item The point representing $C'_u$ in $G_t^{(k)}$ may get an earlier birth time in $\tilde G_t^{(k)}$, in the interval $[w_0, \tau_k(u))$.
\end{enumerate}
For a given $r,s$, the first change might decrease $\beta_k^{r,s}$ by $1$, while the second change might increase it by $1$. In any case, the total difference between $\beta_0^{r,s}(\tilde  G^{(k)})$ and $\beta_0^{r,s}(G^{(k)})$ is no more than one.

\noindent {\bf Case II:} $|C_u|\ge k$, and $|C_v| < k$. Defining $C'_v$ the same as above, note in this case the filtration value of all points outside $C'_v$ will remain unchanged by adding $e_0$. The only change that will occur in this case is that the point in the diagram of $G_t^{(k)}$ corresponding to $C'_v$ will be removed in $\tilde G_t^{(k)}$, since it is now merged with $C_u$. Therefore, the difference in the persistent Betti numbers is at most 1.

\noindent {\bf Case III:} Both $|C_u|\ge k$ and $|C_v|\ge k$. If $C_u=C_v$ then adding $e_0$ creates a $1$-cycle (loop) and does not affect the $k$-cluster persistence diagram.
If $C_u\ne C_v$, then
 both $C_u$ and $C_v$ are represented by different points in the persistence diagram of $G_t^{(k)}$. Adding $e_0$ will cause one of these components to die earlier. In this case $\beta_0^{r,s}$ may be decreased by $1$ (if $s>w_0$).
 
\end{proof}

\bibliographystyle{plain}
\bibliography{zotero}


\end{document}